\documentclass[12pt,a4paper]{amsart}
\usepackage[T1]{fontenc}
\usepackage[hmargin=1.25in,vmargin=1.25in]{geometry}    

\usepackage{amssymb}
\usepackage{mathtools}
\usepackage{verbatim,enumerate}
\usepackage{xcolor}

\usepackage[looser]{newpxtext}
\usepackage[smallerops]{newpxmath}
\usepackage{bm}

\usepackage[hidelinks]{hyperref}

\theoremstyle{definition}
\newtheorem{defn}{Definition}[section]

\newtheorem{rem}[defn]{Remark}

\theoremstyle{plain}
\newtheorem{thm}[defn]{Theorem}
\newtheorem{lem}[defn]{Lemma}
\newtheorem{prop}[defn]{Proposition}
\newtheorem{coro}[defn]{Corollary}

\newcommand{\eps}{\varepsilon} 
 
\newcommand{\bbr}{\mathbb{R}} 
\newcommand{\bbn}{\mathbb{N}}

\newcommand{\vlog}{V^{\log{}}}
\newcommand{\vtlog}{\widetilde{V}^{\log{}}}
\newcommand{\dd}{\mathop{}\!\mathrm{d}}
\DeclareMathOperator{\orb}{Orb}

\title[A note on logarithmic mean equicontinuity] 
{A note on logarithmic mean equicontinuity}

\author[D. Kwietniak]{Dominik Kwietniak}
\address[D. Kwietniak]{Faculty of Mathematics and Computer Science,
Jagiellonian University in Krak\'ow,
ul. {\L}ojasiewicza 6, 30-348 Kraków, Poland}
\email{dominik.kwietniak@uj.edu.pl}
\urladdr{https://orcid.org/0000-0002-7794-2835}

\author[J. Li]{Jian Li}
\address[J. Li]{Institute for  Mathematical Sciences and Artificial Intelligence \& Department of Mathematics,
	Shantou University, Shantou, 515821, Guangdong, China}
\email{lijian09@mail.ustc.edu.cn}
\urladdr{https://orcid.org/0000-0002-8724-3050}

\author[H. Pourmand]{Habibeh Pourmand}
\address[H. Pourmand]{Faculty of Mathematics and Computer Science,
Jagiellonian University in Krak\'ow,
ul. {\L}ojasiewicza 6, 30-348 Kraków, Poland}
\email{habibeh.pourmand@gmail.com}



\date{\today}

\begin{document}

\begin{abstract}
We study the set of harmonic limits of empirical measures in topological dynamical systems. We obtain a characterization
of unique ergodicity based on logarithmic (harmonic) mean convergence in place of Ces\'{a}ro convergence. 
We introduce logarithmic mean equicontinuity and show that a topological dynamical system is logarithmically mean equicontinuous if and only if it is mean equicontinuous.
\end{abstract}

\maketitle 

\section{Introduction}

By a topological dynamical system, we mean a pair $(X,T)$
where $(X,d)$ is a compact metric space and $T\colon X\to X$ is a continuous map.
We say that a topological dynamical system $(X,T)$ is \emph{equicontinuous} if for every $\eps>0$ there exists $\delta>0$ such that for any $x,y\in X$ with $d(x,y)<\delta$, one has $d(T^nx,T^ny)<\eps$ for all $n\geq 0$.
Equicontinuity ensures stable, predictable behaviour as nearby points remain close under all iterations. 
Equicontinuous systems play a fundamental role in topological dynamics, and every surjective equicontinuous system is conjugate to an isometric system given by a rotation on a compact topological group. 
When studying dynamical systems with discrete spectrum, Fomin \cite{F51} introduced a weak form of equicontinuity, called \emph{mean-L-stability}. A topological dynamical system $(X,T)$ is mean-L-stable if for every $\eps>0$ there exists $\delta>0$ such that for any $x,y\in X$ with $d(x,y)<\delta$ and for any $n\geq 0$ outside a set of upper density less than $\eps$, one has $d(T^nx,T^ny)<\eps$. 
It had been an open problem whether any ergodic invariant measure on a mean-L-stable system has discrete spectrum, see \cite{S82}. 
In \cite{LTY15}, Li et al.\  introduced the concept of mean equicontinuity. A topological dynamical system $(X,T)$ is \emph{mean equicontinuous} if for every $\eps>0$ there exists  $\delta>0$ such that for any $x,y\in X$ with $d(x,y)<\delta$, one has  
\[
\limsup_{n\to\infty}\frac{1}{n}\sum_{k=1}^n d(T^{k-1}x,T^{k-1}y)<\eps.
\]
The authors of \cite{LTY15} also showed that mean equicontinuity is equivalent to mean-L-stability and confirmed that every ergodic invariant measure of a mean-L-stable system has discrete spectrum. 
In \cite{G17}, García-Ramos, 
introduced the concept of mean equicontinuity for an invariant measure of a topological dynamical system. He showed that an ergodic invariant measure is 
mean equicontinuous if and only if it has discrete spectrum.
Recently in \cite{HLTXY21}, Huang et al.\  showed that any (not necessarily ergodic) invariant measure is mean equicontinuous if and only if it has discrete spectrum. A direct consequence of this result is that any invariant measure on a mean-L-stable system has discrete spectrum.

In 2010, Sarnak \cite{S12} formulated the following conjecture:
for every topological dynamical system $(X,T)$ with zero topological entropy, every $x\in X$, and every continuous function $f\colon X\to \bbr$,
\[
    \lim_{n\to\infty}\frac{1}{n}\sum_{k=1}^n f(T^kx) {\bm \mu} (n)  =0,
\]
where ${\bm \mu} (n) $ is the classical Möbius function.
We refer the reader to the survey \cite{FKL18} and references therein for the recent developments concerning this conjecture.
In \cite{T17},  Tao studied the following logarithmic version of Sarnak's conjecture:
for every topological dynamical system $(X,T)$ with zero topological entropy, every $x\in X$, and every continuous function $f\colon X\to \bbr$, one has
\[
    \lim_{n\to\infty} \frac{1}{\log n}\sum_{k=1}^n \frac{1}{k} f(T^kx) {\bm \mu} (n)  =0.
\]
Tao proved 
that the logarithmic versions of Sarnak’s and Chowla's conjectures are equivalent. 
In \cite{GKL18} Gomilko et al.\  showed that Sarnak’s conjecture implies the Chowla's conjecture along a subsequence.
Frantzikinakis and Host \cite{FH18} verified the logarithmic version of Sarnak's conjecture for a large class of topological dynamical systems, in particular for all uniquely ergodic systems with zero entropy. 

In view of these results, it is natural to consider the logarithmic version of mean equicontinuity. We say that 
a topological dynamical system is 
\emph{logarithmically mean equicontinuous} if for any $\eps>0$ there exists  $\delta>0$ such that for any $x,y\in X$ with $d(x,y)<\delta$, one has 
\[
    \limsup_{n\to\infty} \frac{1}{H_n}\sum_{k=1}^{n}\frac{1}{k}d(T^{k-1}x,T^{k-1}y)<\eps,
\]
where $H_n=\sum_{k=1}^n \frac{1}{k}$. 
It is easy to see that mean equicontinuity implies logarithmic mean equicontinuity.
The main result of this note is to show that the converse is also true, and hence logarithmic mean equicontinuity is in fact equivalent to mean equicontinuity.

\begin{thm}\label{thm:main-result}
A topological dynamical system $(X,T)$ is logarithmically mean equicontinuous if and only if it is mean equicontinuous.
\end{thm}
The proof uses properties of the harmonic limits of empirical measures studied in \cite{FH18}, \cite{GKL18},  and \cite{P24}.  In order to get our proof we need to
study limits of empirical measures with logarithmic (harmonic) mean convergence replacing Ces\`aro convergence. That is, we study convergence of subsequences of logarithmically averaged empirical measures and compare them with arithmetically averaged empirical measures.  
In particular, we obtain a characterization of unique ergodicity by logarithmic mean convergence, see Theorem~\ref{thm:unique-erg-log-mean}.
We also consider the logarithmic version of mean equicontinuity for invariant measures, and show that an ergodic measure is logarithmically mean equicontinuous if and only if it is mean equicontinuous, see Theorem~\ref{thm:mu-log-mean-eq}.
Note that in general, the results of logarithmic averaging and Ces\`aro averaging of empirical measures are different (see \cite{GKL18}). 

\section{Invariant measures and unique ergodicity}

In this section, we first recall results concerning the classical empirical and invariant measures of topological dynamical systems and a characterization of unique ergodicity. 
Then we show that similar results hold for harmonic limits of empirical measures and we obtain a characterization of unique ergodicity by logarithmic mean convergence.
We think that this approach is of independent interest.

Let $(X,d)$ be a compact metric space and $C(X,\bbr)$ be the space of continuous maps from $X$ to $\bbr$ with the supremum norm $\Vert \cdot \Vert_\infty$. 
Let $2^X$ be the hyperspace of $X$, that is, $2^X$ is the collection of non-empty closed subsets of $X$.
The Hausdorff metric on $2^X$ is defined by 
\[
    d_H(A,B)=\inf\{\eps>0\colon A\subset\{y\in X : d(y,B)<\eps\}\ \&\ B\subset\{x\in X : d(x,A)<\eps\}\}.
\]
Here $d(z,C)$ for $z\in X$ and $C\in 2^X$ means $\inf\{d(z,c):c\in C\}$.
Then $(2^X, d_H)$ is a compact metric space.

Let $\mathcal{B}(X)$ be the collection of all Borel subsets of $X$ and $M(X)$ be the collection of all probability measures on $(X,\mathcal{B}(X))$.
The weak$^*$-topology on $M(X)$ is the smallest topology making each of the maps $\mu\mapsto \int f \dd\mu$ ($f\in C(X,\bbr)$) continuous.
Then $M(X)$ is a compact metrizable space. To define a metric on $M(X)$, one
fixes a countable collection $\{f_j\colon j\geq 1\}$ in $C(X,\bbr)$ such that the linear span of $\{f_j\colon j\geq 1\}$ is dense in $C(X,\bbr)$ and for each $j\geq 1$ we have 
$\Vert f\Vert_\infty \leq 1$ 
and $|f_j(x)-f_j(y)|\leq d(x,y)$ for all $x,y\in X$, see e.g.\ \cite[Theorem~11.2.4]{D02}.
Then for $\mu,\nu\in M(X)$ a compatible metric on $M(X)$ is given by
\[
   \rho(\mu,\nu)=\sum_{j\geq 1}\frac{1}{2^j}\biggl|\int f_j \dd\mu- \int f_j\dd \nu\biggr|.
\]

Let $(X,T)$ be a topological dynamical system. 
We say that $\mu\in M(X)$ is \emph{$T$-invariant} if $\mu(B)=\mu(T^{-1}B)$ for every $B\in\mathcal{B}(X)$. If a $T$-invariant measure $\mu$ is such that for every $B\in\mathcal{B}(X)$ with $T^{-1}B=B$ we have either $\mu(B)=0$ or $\mu(B)=1$, then we say that $\mu$ is \emph{ergodic}.
Let $M_T(X)$ and $M^{\mathrm{erg}}_T(X)$ stand for the collection of $T$-invariant measures and ergodic measures, respectively.

Given $x\in X$, we define the \emph{empirical measure along the orbit of $x$} to be an arithmetic average of Dirac measures concentrated on the consecutive points along the orbit, that is, the empirical measure is the measure
\[\frac{1}{n-m}\sum_{k=m+1}^{n}\delta_{T^{k-1}x},
\] 
where $n>m\ge 0$. We consider the following two sets containing all possible Cesàro limits of convergent subsequences of the sequences of empirical measures generated along the orbit of $x$.
We denote 
\[
    V_T(x)=\biggl\{\mu\in M(X)\colon \exists n_i\nearrow \infty 
    \text{ s.t. } \frac{1}{n_i}\sum_{k=1}^{n_i}\delta_{T^{k-1}x} \to \mu\text{ as }i\to\infty \biggr\},
\] 
and 
\[
     \widetilde{V}_T(x)=\biggl\{\mu\in M(X)\colon 
     \exists n_i-m_i\nearrow \infty 
    \text{ s.t. } \frac{1}{n_i-m_i}\sum_{k=m_i+1}^{n_i}\delta_{T^{k-1}x} \to \mu\text{ as }i\to\infty \biggr\}.
\]

The following result follows from the proof of the Kryloff-Bogoliouboff theorem, see e.g.\ \cite[Theorem 4.1]{EW11}.
We write $\overline{V}$ to denote the closure of $V\subset M(X)$ with respect to the weak$^*$ topology on $M(X)$. 

\begin{lem} \label{lem:vtx}
Let $(X,T)$ be a topological dynamical system. If $x\in X$, then 
\[
    \emptyset \neq V_T(x)=\overline{V_T(x)}\subset \widetilde{V}_T(x)=\overline{\widetilde{V}_T(x)}\subset M_T(X).
\]
\end{lem}

If $V_T(x)=\{\mu\}$, then we say that $x\in X$ is a \emph{generic point} for $\mu$.
If $\mu\in V_T(x)$, then we say that $x\in X$ is \emph{quasi-generic} for $\mu$. Non-ergodic measures may have no generic points, while generic points always exist for ergodic measures. 

\begin{lem}[{see e.g.\ \cite[Corollary 4.20]{EW11}}] \label{lem:ergodic-generic-points}
Let $(X,T)$ be a topological dynamical system and $\mu\in M^{\mathrm{erg}}_T(X)$.
Then $\mu$-almost every point in $X$ is generic for $\mu$.
\end{lem}

If $M_T(X)$ is a singleton, we say that $(X,T)$
is \emph{uniquely ergodic}. Since $M^{\mathrm{erg}}_T(X)$ is the collection of extreme points of $M_T(X)$, the system
$(X,T)$ is uniquely ergodic if and only if $M^{\mathrm{erg}}_T(X)$ is a singleton.

The following result is folklore, it follows from Lemma~\ref{lem:vtx} combined with the definition of unique ergodicity and the definitions of sets $V_T(x)$ and $\widetilde{V}_T(x)$.

\begin{prop}\label{prop:unique-erg-vtx}
Let $(X,T)$ be a topological dynamical system. Then the following assertions are equivalent:
\begin{enumerate}
    \item $(X,T)$ is uniquely ergodic;
    \item there exists  $\mu\in M_T(X)$ such that 
    for every $x\in X$, $V_T(x)=\{\mu\}$;
    \item there exists  $\mu\in M_T(X)$ such that 
    for every $x\in X$, $\widetilde{V}_T(x)=\{\mu\}$.
\end{enumerate}
\end{prop}

The following theorem is a variant of the standard Oxtoby's characterization of unique ergodicity.

\begin{thm}\label{thm:unique-erg} 
Let $(X,T)$ be a topological dynamical system. Then the following assertions are equivalent:
\begin{enumerate}
    \item $(X,T)$ is uniquely ergodic;
    \item there exists  $\mu\in M_T(X)$ such that 
     for any $f\in C(X,\bbr)$  and $x\in X$ we have 
    \[
        \lim_{n\to\infty}\frac{1}{n} \sum_{k=1}^n f(T^{k-1}x)=\int f\dd\mu;
    \]
    \item for any $f\in C(X,\bbr)$, the functions $x\mapsto \frac{1}{n} \sum_{k=1}^n f(T^{k-1}x)$ converge pointwise to a constant function as $n\to\infty$;
        \item for any $f\in C(X,\bbr)$, the functions $x\mapsto \frac{1}{n} \sum_{k=1}^n f(T^{k-1}x)$ converge uniformly  to a constant function as $n\to\infty$;
     \item there exists  $\mu\in M_T(X)$ such that 
     for any $f\in C(X,\bbr)$  and $x\in X$ we have 
    \[
        \lim_{n-m\to\infty} \frac{1}{n-m} \sum_{k=m+1}^n f(T^{k-1}x)= \int f\dd\mu;
    \]
    \item  for any $f\in C(X,\bbr)$, the functions
    $x\mapsto \frac{1}{n-m} \sum_{k=m+1}^n f(T^{k-1}x)$ converge pointwise  to a constant function as $n-m\to\infty$;
    \item  for any $f\in C(X,\bbr)$, the functions
    $x\mapsto \frac{1}{n-m} \sum_{k=m+1}^n f(T^{k-1}x)$ converge uniformly to a constant function as $n-m\to\infty$.
\end{enumerate}
In addition, the constant function mentioned in assertions (3),~(4),~(6) and (7) is always the function $x\mapsto\int f\dd\mu$, where $\mu$ is the unique $T$-invariant measure. 
\end{thm}
\begin{proof}
    The equivalence of assertions (1)--(4) follows from \cite[\S 5.3, p.~124]{O52}, see also \cite[Theorem 4.10]{EW11}.

    The equivalence (1)$\Leftrightarrow$(5) follows from Proposition~\ref{prop:unique-erg-vtx} and the definition of the compatible metric $\rho$ on $M(X)$. 

    The implications (7)$\Rightarrow$(6)$\Rightarrow$(3) are obvious.

    (4)$\Rightarrow$(7). Fix $f\in C(X,\bbr)$. It follows from (4) that 
    for some $c\in\bbr$, for every $\eps>0$, there exists $N\in\bbn$ so that for any $n\geq N$ it holds 
    \[
    \sup_{x\in X}\biggl|\frac{1}{n} \sum_{k=1}^nf(T^{k-1}x) -c\biggr|<\eps.
    \]
    Hence, if $n-m\geq N$, then
    \[
    \sup_{x\in X}\biggl|\frac{1}{n-m} \sum_{k=m+1}^n f(T^{k-1}x) -c\biggr|
    =\sup_{x\in X}\biggl|\frac{1}{n-m} \sum_{k=1}^{n-m}f(T^{k-1}(T^mx) -c\biggr|<\eps.
    \]
    Thus, the functions
    $x\mapsto \frac{1}{n-m} \sum_{k=m+1}^n f(T^{k-1}x)$ converge uniformly to a constant function as $n-m\to\infty$.
\end{proof}

\begin{rem}
As noticed in \cite[page 126]{O52}, the equivalent assertions presented in Theorem~\ref{thm:unique-erg} do not hold for noncontinuous functions. 
Let $X=\mathbb{R}/\mathbb{Z}$ be the unit circle and $T\colon X\to X$ be given by $x\mapsto x+\alpha\pmod 1$ for some irrational $\alpha$.
Then the Lebesgue measure $m$ is the unique $T$-invariant measure on $X$.
Let $U=\bigcup_{n\geq 1} (n\alpha-\frac{1}{2^{n+1}},n\alpha+
\frac{1}{2^{n+1}})$. 
Then $m(U)\leq \frac{1}{2}$.
Let $f$ be the characteristic function of $U$.
By the Birkhoff ergodic theorem, there exists a measurable subset $X_0$ of $X$ with $m(X_0)=1$ such that for any $x\in X_0$ we have 
 \[
        \lim_{n\to\infty}\frac{1}{n} \sum_{k=1}^n f(T^{k-1}x)=m(U).
 \]
It is clear that $\lim_{n\to\infty}\frac{1}{n} \sum_{k=1}^n f(T^{k-1}0)=1$. Hence, $0\not\in X_0$, so the convergence is not uniform on $X$.
Furthermore, for each $n\geq 1$, there exists a nondegenerate interval $Z_n$ in $X$
such that for any $z\in Z_n$ we have $\frac{1}{n}\sum_{k=1}^n f(T^{k-1}x)=1$.
So the convergence is also not uniform on the set $X_0$.
\end{rem} 

Recall that $H_n=\sum_{k=1}^n \frac{1}{k}$ for each $n\in\bbn$.  
Note that $\lim_{n\to\infty}\frac{H_n}{\log n} =1$. 

Using the summation by parts trick, for any sequence $(x_n)$ in $\bbr$ and $n\ge 1$, one has the following classical relation between Cesàro averages and harmonic averages of the sequence: 
\begin{align*}
    \frac{1}{H_n} \sum_{k=1}^n \frac{1}{k}x_k &=
    \frac{1}{H_n}\biggl(x_1+\sum_{k=2}^n\frac{1}{k}\biggl(\sum_{i=1}^kx_i -\sum_{i=1}^{k-1}x_i\biggr)\biggr)\\
    &=\frac{1}{H_n}\biggl(\frac{1}{n}\sum_{k=1}^n x_k +\sum_{k=1}^{n-1}\frac{1}{(k+1)k}\sum_{i=1}^kx_i\biggr) \\
&= \frac{1}{H_n}\biggl(\frac{1}{n}\sum_{k=1}^n x_k\biggr) +
   \frac{1}{H_n} \sum_{k=1}^{n-1}\frac{1}{k+1} \biggl(\frac{1}{k}\sum_{i=1}^kx_i\biggr).
\end{align*}
Then we have the following folklore lemma.

\begin{lem}\label{lem:log-mean-seq}
Let $(x_n)$ be a bounded sequence in $\bbr$. 
Then
\[
\liminf_{n\to\infty} \frac{1}{n}\sum_{k=1}^n x_k \leq \liminf_{n\to\infty} \frac{1}{H_n} \sum_{k=1}^n \frac{1}{k}x_k\leq \limsup_{n\to\infty} \frac{1}{H_n} \sum_{k=1}^n \frac{1}{k}x_k \leq \limsup_{n\to\infty}\frac{1}{n}\sum_{k=1}^n x_k.
\]
\end{lem}

Let $(X,T)$ be a topological dynamical system. 
To define \emph{logarithmic empirical measures}, we take harmonic average of Dirac measures concentrated on sequence of consecutive points along the orbit of $x\in X$, that is, the logarithmic empirical measure is the measure
\[\frac{1}{H_{n-m}}\sum_{k=m+1}^{n}\frac{1}{k-m}\delta_{T^{k-1}x},
\] 
where $n>m\ge 0$. We now consider sets containing all possible accumulation points of the sequences of empirical measures generated along the orbit of $x$.
For $x\in X$, we define 
\[
 \vlog_T(x)=\biggl\{\mu\in M(X)\colon \exists n_i\nearrow \infty 
    \text{ s.t. } \frac{1}{H_{n_i}}\sum_{k=1}^{n_i} \frac{1}{k}\delta_{T^{k-1}x} \to \mu\text{ as }i\to\infty \biggr\}
\]
and 
\[
 \vtlog_T(x)=\biggl\{\mu\in M(X)\colon \exists n_i-m_i\nearrow \infty 
    \text{ s.t. } \frac{1}{H_{n_i-m_i}}\sum_{k=m_i+1}^{n_i} \frac{1}{k-m_i}\delta_{T^{k-1}x} \to \mu\text{ as }i\to\infty \biggr\}.
\]
If $\vlog_T(x)=\{\mu\}$, then we say that $x\in X$ is \emph{logarithmically generic} for $\mu$.
If $\mu\in \vlog(x)$, then we say that $x\in X$ is \emph{logarithmically quasi-generic} for $\mu$.

The proof of the following lemma is similar to that of \cite[Theorem 4.1]{EW11}, but we provide it for the sake of completeness. 

\begin{lem}\label{lem:vlogT-seq}
Let $(X,T)$ be a topological dynamical system. If $(x_m)_{m=1}^\infty$ is a sequence in $X$ and $\mu\in M(X)$ is such that
\[
\mu=\lim_{m\to\infty} \frac{1}{H_{n_m}}\sum_{k=1}^{n_m} \frac{1}{k}\delta_{T^{k-1}x_m},
\]
where $n_m\nearrow\infty$ as $m\to\infty$, then $\mu \in M_T(X)$.
\end{lem}
\begin{proof} 
Fix  $f\in C(X,\bbr)$. 
We have
\[
\biggl| \int f\dd\mu- \int f\circ T\dd\mu\biggr | = 
\biggl| 
\lim_{m\to\infty} \frac{1}{{H_{n_m}}} 
    \sum_{k=1}^{n_m}\frac{1}{k} \left(f(T^{k-1}x_m)- f(T^kx_m)\right) \biggr|.
\]
For every $m\ge 1$ it also holds
\begin{multline*}\biggl| 
\frac{1}{{H_{n_m}}} 
    \sum_{k=1}^{n_m}\frac{1}{k} \left(f(T^{k-1}x_m)- f(T^kx_m)\right) \biggr|\leq\\
\frac{1}{{H_{n_m}}}   \biggl| f(x_m) -\sum_{k=1}^{n_m-1} \frac{1}{k(k+1)} f(T^k x_m)-\frac{1}{n_m}f(T^{n_m}x_m) \biggr|\leq\\
\frac{1}{{H_{n_m}}} \biggl(2+\sum_{k=1}^{n_m}\frac{1}{k(k+1)}\biggr)\Vert f\Vert_\infty .
\end{multline*}
Finally, note that
\[
\lim_{m\to\infty} \frac{1}{{H_{n_m}}} \biggl(2+\sum_{k=1}^{n_m}\frac{1}{k(k+1)}\biggr)\Vert f\Vert_\infty =0.
\]
Therefore $\int f\dd\mu= \int f\circ T\dd\mu$ for every $f\in C(X,\bbr)$. Hence $\mu$ is an invariant measure. 
\end{proof}

As a consequence, we obtain an analogue of Lemma~\ref{lem:vtx}.

\begin{lem}\label{lem:vlogT-x}
Let $(X,T)$ be a topological dynamical system. Then for any $x\in X$ the sets $\vlog_T(x)$ and $\vtlog_T(x)$ are closed in $M(X)$ and
\[
    \emptyset \neq \vlog_T(x)\subset \vtlog_T(x)\subset M_T(X).
\]
\end{lem}
\begin{proof}
Fix $x\in X$. It is clear that $\vlog_T(x)\subset \vtlog_T(x)$ and that both sets are closed in $M(X)$. By compactness of $M(X)$, one has $\vlog_T(x)\neq\emptyset$.
It remains to show that $\vtlog_T(x)\subset M_T(X)$.
Let $\mu\in\vtlog_T(x)$. There there exist two sequences $(n_i)$ and $(m_i)$ in $\bbn$ with $n_i-m_i\nearrow \infty $ as $i\to\infty$ such that 
\[
\lim_{i\to\infty} \frac{1}{H_{n_i-m_i}}\sum_{k=m_i+1}^{n_i} \frac{1}{k-m_i}\delta_{T^{k-1}x} = \mu.
\]
We finish the proof applying Lemma \ref{lem:vlogT-seq} to sequences $y_i=T^{m_i}x$ and $l_i=n_i-m_i$.
\end{proof}

\begin{rem}
By the proof of Lemma~\ref{lem:vlogT-seq}, for every $x\in X$ we have
$\vlog_T(x)=\vlog_T(Tx)$ and $\vtlog_T(x)=\vtlog_T(Tx)$.
\end{rem}

By Lemma~\ref{lem:log-mean-seq}, it is easy to see that
if $V_T(x)=\{\mu\}$, then $\vlog_T(x)=\{\mu\}$. This was noted in \cite{GKL18}. In fact, as a consequence of the summation by parts trick, we have the following general result.

\begin{prop}[{\cite[Proposition 2.1]{GKL18}}] \label{prop:vlogtx-vtx}
Let $(X,T)$ be a topological dynamical system and $x\in X$.
Then $\vlog_T(x)$ is contained in the closed convex hull of $V_T(x)$. 
\end{prop}

\begin{rem}
Note that the sets $V_T(x)$ and $\vlog_T(x)$ can be disjoint, see \cite[Proposition 12.13]{GKL18}.
\end{rem}

\begin{prop}\label{prop:ergodic-log-mean}
Let $(X,T)$ be a topological dynamical system and $\mu\in M_T(X)$. Then 
the following assertions are equivalent:
\begin{enumerate}
    \item $\mu$ is ergodic;
    \item for any $f\in L^1(X,\mu)$, we have $\frac{1}{H_n} \sum_{k=1}^n \frac{1}{k} f(T^{k-1}x)\to \int f\dd \mu$ as $n\to\infty$ for $\mu$-a.e. $x\in X$;
    \item for any  $f\in L^1(X,\mu)$, the sequence of $L^1$-functions $\frac{1}{H_n} \sum_{k=1}^n \frac{1}{k} f\circ T^{k-1}$ converges to the constant function $\int f \dd\mu$ as $n\to\infty$ in $L^1(X,\mu)$.
\end{enumerate} 
\end{prop}
\begin{proof}
(1)$\Rightarrow$(2). 
Fix $f\in L^1(X,\mu)$. 
By the Birkhoff ergodic theorem, see e.g.\ \cite[Theorem 2.30]{EW11}, we have $\frac{1}{n} \sum_{k=1}^n f(T^{k-1}x)\to \int f\dd\mu$ as $n\to\infty$ for $\mu$-a.e. $x\in X$.
It follows from Lemma~\ref{lem:log-mean-seq} that 
$\frac{1}{H_n} \sum_{k=1}^n \frac{1}{k} f(T^{k-1}x)\to \int f\dd \mu$ as $n\to\infty$ for $\mu$-a.e. $x\in X$.

(2)$\Rightarrow$(3). By the Lebesgue-dominated convergence theorem,
(3) holds for any $f\in L^\infty(X,\mu)$.
As $L^\infty(X,\mu)$ is dense in $L^1(X,\mu)$, (3) holds for all $f\in L^1(X,\mu)$.

(3)$\Rightarrow$(1).
Let $B\in\mathcal{B}(X)$ with $T^{-1}B=B$.
Let $f=\mathbf{1}_B$. Then $f\in L^1(X,\mu)$ and $f\circ T=f$. Furthermore, $\frac{1}{H_n} \sum_{k=1}^n \frac{1}{k} f(T^{k-1}x) = f(x)$ for any $n\in\bbn$ and $x\in X$.
Since $\Vert f -\int f\dd\mu\Vert_1=0$, we see that $f\equiv 1$ or $f\equiv 0$ up to a $\mu$-null set. It follows that $\mu(B)=1$ or $\mu(B)=0$.
This shows that $\mu$ is ergodic.
\end{proof}

The following result is a special case of \cite[Proposition 5.3.1]{P24}. Here we provide a direct proof for the sake of completeness.

\begin{prop}\label{prop:ergodic-vlog}
Let $(X,T)$ be a topological dynamical system and $\mu\in M_T(X)$. 
Then $\mu$ is ergodic if and only if for $\mu$-a.e.\ $x\in X$ we have $\vlog_T(x)=\{\mu\}$.
\end{prop}
\begin{proof}
($\Rightarrow$) Assume that $\mu$ is ergodic. 
By Lemma~\ref{lem:ergodic-generic-points}, it holds $V_T(x)=\{\mu\}$ for $\mu$-a.e.\ $x\in X$.  
Using Proposition~\ref{prop:vlogtx-vtx} we see that $\vlog_T(x)=\{\mu\}$ for $\mu$-a.e.\ $x\in X$. 

($\Leftarrow$) There exists a Borel subset $X_0$ of $X$ with $\mu(X_0)=1$ such that $\vlog_T(x)=\{\mu\}$ for all $x\in X_0$.
By the definition of weak$^*$-topology on $M(X)$, for any $f\in C(X,\bbr)$ and $x\in X_0$, we have 
$\frac{1}{H_n} \sum_{k=1}^n \frac{1}{k} f(T^{k-1}x)\to \int f\dd \mu$ as $n\to\infty$.
By the Lebesgue dominated convergence theorem, for every $f\in C(X,\bbr)$ we have
\[
  \biggl\Vert \frac{1}{H_n} \sum_{k=1}^n \frac{1}{k} f\circ T^{k-1} -\int f\dd \mu  \biggr\Vert_1 \to 0, \text{ as }n\to\infty.
\]
Since $C(X,\bbr)$ is dense in $L^1(X,\mu)$, for any $f\in L^1(X,\mu)$ we have 
\[
  \biggl\Vert \frac{1}{H_n} \sum_{k=1}^n \frac{1}{k} f\circ T^{k-1} -\int f\dd \mu  \biggr\Vert_1 \to 0, \text{ as }n\to\infty.
\]
Now by Proposition~\ref{prop:ergodic-log-mean}, $\mu$ is ergodic.
\end{proof}

\begin{prop}\label{prop:uniqlue-ergodic-vlogt}
Let $(X,T)$ be a topological dynamical system. Then the following assertions are equivalent:
\begin{enumerate}
    \item $(X,T)$ is uniquely ergodic;
    \item there exists  $\mu\in M_T(X)$ such that 
    for every $x\in X$ we have $\vlog_T(x)=\{\mu\}$;
    \item there exists  $\mu\in M_T(X)$ such that 
    for every $x\in X$ we have $\vtlog_T(x)=\{\mu\}$.
\end{enumerate}
\end{prop}
\begin{proof}
(1)$\Rightarrow$(3). As $(X,T)$ is uniquely ergodic,
there exists $\mu\in M(X)$ such that $M_T(X)=\{\mu\}$.
It follows from Lemma~\ref{lem:vlogT-x} that  $\vtlog_T(x)=\{\mu\}$ for every $x\in X$.

(3)$\Rightarrow$(2). It is a consequence of 
$\emptyset\neq \vlog_T(x)\subset \vtlog_T(x)$, see Lemma \ref{lem:vlogT-x}.

(2)$\Rightarrow$(1). Fix $\nu\in M_T^{\mathrm{erg}}(X)$. By Proposition~\ref{prop:ergodic-vlog}, there exists $x\in X$ such that $\vlog_T(x)=\{\nu\}$.
Using (2), one has $\mu=\nu$. This shows that $M_T^{\mathrm{erg}}(X)$ is a singleton.
Then $(X,T)$ is uniquely ergodic.
\end{proof}

Now we have the following characterization of unique ergodicity by logarithmic mean convergence. 

\begin{thm} \label{thm:unique-erg-log-mean}
Let $(X,T)$ be a topological dynamical system. Then the following assertions are equivalent:
\begin{enumerate}
    \item $(X,T)$ is uniquely ergodic;
    \item there exists  $\mu\in M_T(X)$ such that 
     for any $f\in C(X,\bbr)$ and $x\in X$ one has 
    \[
        \lim_{n-m\to\infty }
        \frac{1}{H_{n-m}} \sum_{k=m+1}^n \frac{1}{k-m}f(T^{k-1}x)= \int f\dd\mu;
    \]
    \item there exists  $\mu\in M_T(X)$ such that 
     for any $f\in C(X,\bbr)$ and $x\in X$ one has 
    \[
        \lim_{n\to\infty} \frac{1}{H_n} \sum_{k=1}^n \frac{1}{k}f(T^{k-1}x)= \int f\dd\mu;
    \]
    \item for any $f\in C(X,\bbr)$ the functions $\frac{1}{H_n} \sum_{k=1}^n \frac{1}{k} f\circ T^{k-1}$  converge pointwise to a constant function as $n\to\infty$;
    \item  for any $f\in C(X,\bbr)$ the functions
    $\frac{1}{H_{n-m}} \sum_{k=m+1}^n \frac{1}{k-m}f\circ T^{k-1}$ converge pointwise to a constant function as $n-m\to\infty$;
    \item for any $f\in C(X,\bbr)$ the functions $\frac{1}{H_n} \sum_{k=1}^n \frac{1}{k} f\circ T^{k-1}$ uniformly converge to a constant function as $n\to\infty$;
    \item  for any $f\in C(X,\bbr)$ the functions
    $\frac{1}{H_{n-m}} \sum_{k=m+1}^n \frac{1}{k-m} f\circ T^{k-1}$ converge uniformly to a constant function as $n-m\to\infty$.
\end{enumerate}
\end{thm}

\begin{proof}
Combining Proposition~\ref{prop:uniqlue-ergodic-vlogt} with the definition of weak$^*$-topology on $M(X)$ and definitions of sets $\vlog_T(x)$ and $\vtlog_T(x)$, we obtain (1)$\Leftrightarrow$(2)$\Leftrightarrow$(3).  It is also clear that (2)$\Rightarrow$(5)$\Rightarrow$ (4) and (7)$\Rightarrow$(6)$\Rightarrow$(4). 

(4)$\Rightarrow$(1).
Fix $f\in C(X,\bbr)$. Let $c_f$ be the constant function such that $\frac{1}{H_n} \sum_{k=1}^n \frac{1}{k} f\circ T^{k-1}$  converges pointwise to $c_f$ as $n\to\infty$.
By the Lebesgue dominated convergence theorem for every $\mu \in M_T(X)$ and $f\in C(X,\bbr)$ one has
\[
    \int f\dd\mu = \int \lim_{n\to\infty} \frac{1}{H_n} \sum_{k=1}^n \frac{1}{k}f(T^{k-1}x) \dd\mu=c_f.
\] 
Since continuous functions separate measures on $X$, this is possible only if $M_T(X)$ is a singleton.

(3)$\Rightarrow$(6).
If there exists $f\in C(X,\bbr)$ such that the convergence in (3) is not uniform, then there exists $\varepsilon>0$ such that for any $N_0\in\bbn$ there are $n>N_0$ and $x_n\in X$ such that 
\[
    \biggl|\frac{1}{H_n} \sum_{k=1}^n \frac{1}{k} f(T^{k-1}x_n)-\int f\dd\mu \biggr|\geq \varepsilon.
\]
Let $\nu_n= \frac{1}{H_n} \sum_{k=1}^n \frac{1}{k}\delta_{T^{k-1}x_n}$, so that $\bigl |\int f\dd\nu_n -\int f\dd\mu\bigr|\geq \varepsilon$.
By compactness of $M(X)$, there exists a sequence $n_i\nearrow \infty$ in $\bbn$ and $\nu\in M(X)$ such that $\nu_{n_i}\to\nu$ as $i\to\infty$.
Then $\bigl |\int f\dd\nu -\int f\dd\mu\bigr|\geq \varepsilon$.
Using Lemma~\ref{lem:vlogT-seq} we conclude that $\nu$ is an invariant measure. This contradicts the unique ergodicity of $(X,T)$.

(6)$\Rightarrow$(7). The proof is basically the same as the proof of (4)$\Rightarrow$(7) in Theorem~\ref{thm:unique-erg}.
\end{proof}

\section{Logarithmic mean equicontinuity}

In this section, we study logarithmic mean equicontinuity. We will use the characterization of unique ergodicity by logarithmic mean convergence to show that logarithmic mean equicontinuity is equivalent to mean equicontinuity. 
Moreover, we show that Weyl logarithmic mean equicontinuity is also equivalent to logarithmic mean equicontinuity. 

We first recall the definition of mean equicontinuity and introduce logarithmic mean equicontinuity.
\begin{defn}
We say that a topological dynamical system $(X,T)$ is \emph{mean equicontinuous} if for any $\eps>0$ there exists  $\delta>0$ such that for any $x,y\in X$ with $d(x,y)<\delta$ one has 
\[
    \limsup_{n\to\infty} \frac{1}{n}\sum_{k=1}^{n}
    d(T^{k-1}x,T^{k-1}y)<\eps.
\]
A topological dynamical system $(X,T)$  is \emph{logarithmically mean equicontinuous} if for any $\eps>0$ there exists  $\delta>0$ such that for any $x,y\in X$ with $d(x,y)<\delta$ one has 
\[
    \limsup_{n\to\infty} \frac{1}{H_n}\sum_{k=1}^{n}\frac{1}{k}d(T^{k-1}x,T^{k-1}y)<\eps.
\]
\end{defn}

\begin{rem}\label{rem:log-mean-eq}
It follows from Lemma~\ref{lem:log-mean-seq}, that if $(X,T)$ is mean equicontinuous, then it is also logarithmically mean equicontinuous.
\end{rem}

The following result is a special case of \cite[Proposition 5.2.13]{P24}.
We provide a proof for the sake of completeness.

\begin{lem}\label{lem:log-mean-equi-vlog}
Let $(X,T)$ be a topological dynamical system.
If $(X,T)$ is logarithmically mean equicontinuous,
then the map $\vlog_T\colon X\to 2^{M(X)}$ given by $x\mapsto \vlog_T(x)$, is continuous.
\end{lem}
\begin{proof}
Fix a countable collection $\{f_j\colon j\geq 1\}$ in $C(X,\bbr)$ such that the linear subspace of $C(X,\bbr)$ spanned by  $\{f_j\colon j\geq 1\}$ is dense in $C(X,\bbr)$. Furthermore, assume that for each $j\geq 1$, 
$\Vert f_j\Vert_\infty \leq 1$ 
and $|f_j(x)-f_j(y)|\leq d(x,y)$ for all $x,y\in X$.
Then for any $j\in\bbn$, $x,y\in X$, and $n\in\bbn$ we have 
\begin{multline}\label{ineq:star}
\biggl|\int f_j \dd \biggl(\frac{1}{H_n}\sum_{k=1}^{n}\frac{1}{k}\delta_{T^{k-1}x}\biggr) -\int f_j \dd \biggl(\frac{1}{H_n}\sum_{k=1}^{n}\frac{1}{k}\delta_{T^{k-1}y}\biggr)\biggr|
\leq\\ \leq\frac{1}{H_n}\sum_{k=1}^{n} \frac{1}{k} |f_j(T^{k-1}x)-f_j(T^{k-1}y)|\leq \frac{1}{H_n}\sum_{k=1}^{n} \frac{1}{k} d(T^{k-1}x,T^{k-1}y).
\end{multline}
As $(X,T)$ is logarithmically mean equicontinuous, 
for any $\eps>0$ there exists  $\delta>0$ such that for any $x,y\in X$ with $d(x,y)<\delta$ one has 
\[
    \limsup_{n\to\infty} \frac{1}{H_n}\sum_{k=1}^{n}\frac{1}{k}d(T^{k-1}x,T^{k-1}y)<\eps.
\]
Fix $x,y\in X$ with $d(x,y)<\delta$. Let $\mu\in \vlog_T(x)$.
There exists a sequence $n_i\nearrow \infty $ such that 
   $\frac{1}{H_{n_i}}\sum_{k=1}^{n_i} \frac{1}{k}\delta_{T^{k-1}x} \to \mu$ as $i\to\infty$.
By the compactness of $M(X)$, without loss of generality, assume that  $\frac{1}{H_{n_i}}\sum_{k=1}^{n_i} \frac{1}{k}\delta_{T^{k-1}y} \to \nu$ as $i\to\infty$. 
Then using \eqref{ineq:star} we get
\begin{align*}
\rho(\mu,\nu) &= \lim_{i\to\infty} \rho\biggl( \frac{1}{H_{n_i}}\sum_{k=1}^{n_i} \frac{1}{k}\delta_{T^{k-1}x}, \frac{1}{H_{n_i}}\sum_{k=1}^{n_i} \frac{1}{k}\delta_{T^{k-1}x}\biggr)\\
&= \lim_{i\to\infty} \sum_{j\geq 1}\frac{1}{2^j} 
\biggl|\int f_j \dd \biggl(\frac{1}{H_{n_i}}\sum_{k=1}^{n_i}\frac{1}{k}\delta_{T^{k-1}x}\biggr) -\int f_j \dd \biggl(\frac{1}{H_{n_i}}\sum_{k=1}^{n_i}\frac{1}{k}\delta_{T^{k-1}y}\biggr)\biggr|\\
&\leq \lim_{n\to\infty} \sup \frac{1}{H_n}\sum_{k=1}^{n} \frac{1}{k} d(T^{k-1}x,T^{k-1}y)<\eps.
\end{align*}
This implies that $\rho\bigl(\mu, \vlog_T(y)\bigr)<\eps$.
By symmetry, one has $\rho\bigl(\nu, \vlog_T(x)\bigr)<\eps$ for any $\nu\in \vlog_T(y)$.
Then $\rho_H(\vlog_T(x),\vlog_T(y))<\eps$.
This shows that that the map $\vlog_T: X\to 2^{M(X)}$ given by $x\mapsto \vlog_T(x)$, is continuous.
\end{proof}

For a point $x\in X$, the orbit of $x$ is the set $\orb(x,T)=\{T^nx\colon n\geq 0\}$.
It is clear that for any $x\in X$ the closure of $\orb(x,T)$ is $T$-invariant, hence $(\overline{\orb(x,T)},T)$ forms a subsystem of $(X,T)$.

\begin{lem}\label{lem:log-mean-eq-orbit-unqiuley-ergodic}
If $(X,T)$ is logarithmically mean equicontinuous,
then for every $x\in X$, the topological dynamical system $(\overline{\orb(x,T)},T)$ is uniquely ergodic.
\end{lem}
\begin{proof}
Fix $x\in X$. 
It is easy to see that the map $\vlog_T$ is constant on $\orb(x,T)$.
By Lemma~\ref{lem:log-mean-equi-vlog}, the map $\vlog_T\colon X\to 2^{M_T(X)}$ is also continuous, so $\vlog_T$ is constant on $\overline{\orb(x,T)}$.
Let $\mu\in M^{\mathrm{erg}}_T\Bigl(\overline{\orb(x,T)}\Bigr)$.
By Proposition~\ref{prop:ergodic-vlog}  
there exists  $y\in \overline{\orb(x,T)}$ such that $\vlog_T(y)=\{\mu\}$. 
Therefore $\vlog_T(z)=\{\mu\}$ for all $z\in \overline{\orb(x,T)}$. Using Proposition~\ref{prop:uniqlue-ergodic-vlogt} we see that $(\overline{\orb(x,T)},T)$ is uniquely ergodic.
\end{proof}

Now we are ready to prove the main result of this section. 
\begin{thm}\label{thm:log-meqn-eq-main-result}
If $(X,T)$ is logarithmically mean equicontinuous, 
then it is mean equicontinuous.
\end{thm}
\begin{proof}
It is easy to see that if $(X,T)$ is logarithmically mean equicontinuous, 
then the product system $(X\times X,T\times X)$ is also logarithmically mean equicontinuous.
By Lemma~\ref{lem:log-mean-eq-orbit-unqiuley-ergodic}, for every $(x,y)\in X\times X$, the topological dynamical system $(\overline{\orb((x,y),T\times T)},T\times T)$ is uniquely ergodic.
Let $\mu_{(x,y)}$ be the unique measure on $(\overline{\orb((x,y),T\times T)},T\times T)$.
Since the metric $d$ is a continuous function on $X\times X$, 
we apply Theorems~\ref{thm:unique-erg}\,(3) and \ref{thm:unique-erg-log-mean}\,(3) to get 
\[ 
\lim_{n\to\infty} \frac{1}{n}\sum_{k=1}^n d(T^{k-1}x,T^{k-1}y) = \int d(x,y) \dd\mu_{(x,y)} =
\lim_{n\to\infty} \frac{1}{H_n}\sum_{k=1}^{n}\frac{1}{k}d(T^{k-1}x,T^{k-1}y) 
\]
for every $(x,y)\in X\times X$.
Then $(X,T)$ is mean equicontinuous, as $(X,T)$ is logarithmically mean equicontinuous.
\end{proof}

Combining Remark~\ref{rem:log-mean-eq} and Theorem~\ref{thm:log-meqn-eq-main-result}, we get Theorem~\ref{thm:main-result}.

\begin{defn}
We say that a topological dynamical system $(X,T)$ is \emph{Weyl mean equicontinuous} if for any $\eps>0$ there exists  $\delta>0$ such that for any $x,y\in X$ with $d(x,y)<\delta$ one has 
\begin{equation}\label{eqn:def_Banach}
    \limsup_{n-m\to\infty} \frac{1}{n-m}\sum_{k=m+1}^{n}d(T^{k-1}x,T^{k-1}y)<\eps.
\end{equation}
\end{defn}

Note that Weyl mean equicontinuity is called Banach mean equicontinuity in \cite{LTY15}, as the formula \eqref{eqn:def_Banach} is similar to the definition of upper Banach density.
We follow \cite{DG16}  in calling this notion Weyl mean equicontinuity  to keep an analogy with the Weyl pseudo-metric.
The authors of \cite{LTY15} asked whether Weyl mean equicontinuity is equivalent to mean equicontinuity.
This was confirmed in \cite{DG16} for minimal systems and in \cite{QZ20} for all topological dynamical systems. Recall that a topological dynamical system $(X,T)$ is \emph{minimal} if for every $x\in X$, the orbit $\orb(x,T)$ is dense in $X$. 

\begin{thm}[\cite{DG16,QZ20}]\label{thm:equiv-W-meq}
Let $(X,T)$ be a topological dynamical system.
Then $(X,T)$ is Weyl mean equicontinuous if and only if it is mean equicontinuous.
\end{thm}

\begin{defn}
We say that a topological dynamical system $(X,T)$ is \emph{Weyl logarithmically mean equicontinuous} if for any $\eps>0$ there exists $\delta>0$ such that for any $x,y\in X$ with $d(x,y)<\delta$, one has
\[
    \limsup_{n-m\to\infty} \frac{1}{H_{n-m}}\sum_{k=m+1}^{n}\frac{1}{k-m}d(T^{k-1}x,T^{k-1}y)<\eps.
\]
\end{defn}
\begin{rem}\label{rem:W-meq}
It follows from Lemma~\ref{lem:log-mean-seq}, that if $(X,T)$ is Weyl mean equicontinuous, then it is also logarithmically Weyl mean equicontinuous.    
\end{rem}

The following result follows easily from the equivalence of logarithmic mean equicontinuity and mean equicontinuity (Theorem \ref{thm:main-result}),
and the equivalence of Weyl mean equicontinuity and mean equicontinuity (Theorem \ref{thm:equiv-W-meq}).
Here we also provide another independent proof.
The idea of the proof also provides a new proof of the following result: a topological dynamical system is mean equicontinuous if and only if it is Weyl mean equicontinuous.

\begin{prop}
Let $(X,T)$ be a topological dynamical system. 
Then $(X,T)$ is logarithmically mean equicontinuous if and only if it is Weyl logarithmically mean equicontinuous.
\end{prop}
\begin{proof}
It follows from the definition that if $(X,T)$  is Weyl logarithmically mean equicontinuous, then it is also  logarithmically mean equicontinuous.

Now assume that  $(X,T)$ is logarithmically mean equicontinuous.
By Lemma~\ref{lem:log-mean-eq-orbit-unqiuley-ergodic}, for every $(x,y)\in X\times X$, $(\overline{\orb((x,y),T\times T)},T\times T)$ is uniquely ergodic.
Let $\mu_{(x,y)}$ be the unique measure on $(\overline{\orb((x,y),T\times T)},T\times T)$.
Since the metric $d$ is a continuous function on $X\times X$, 
by Theorem~\ref{thm:unique-erg-log-mean}\,(2) and (3), for every $(x,y)\in X\times X$ we have 
\begin{align*}
\lim_{n\to\infty}  \frac{1}{H_{n-m}}\sum_{k=m+1}^{n}\frac{1}{k-m}d(T^{k-1}x,T^{k-1}y) &
= \int d(x,y) \dd\mu_{(x,y)} \\
& =
\lim_{n\to\infty} \frac{1}{H_n}\sum_{k=1}^{n}\frac{1}{k}d(T^{k-1}x,T^{k-1}y). 
\end{align*}
for every $(x,y)\in X\times X$.
Then $(X,T)$ is  Weyl logarithmically mean equicontinuous, as $(X,T)$ is logarithmically mean equicontinuous.
\end{proof}

\begin{defn}
We say that a topological dynamical system $(X,T)$ is \emph{mean sensitive} if there exists a constant $\eps>0$ such that for any $x\in X$ and  any $\delta>0$ there exists $y\in X$ with $d(x,y)<\delta$ satisfying
\[
    \limsup_{n\to\infty} \frac{1}{n}\sum_{k=1}^{n}
    d(T^{k-1}x,T^{k-1}y)\geq \eps.
\]
\end{defn}

The following dichotomy result was proved in \cite[Corollary 5.5]{LTY15}.

\begin{thm}\label{thm:dich-mean-eq}
Let $(X,T)$ be a minimal system.
Then $(X,T)$ is either mean equicontinuous or mean sensitive.
\end{thm}

\begin{defn}
We say that a topological dynamical system $(X,T)$ is \emph{logarithmically mean sensitive} if there exists a constant $\eps>0$ such that for any $x\in X$ and any $\delta>0$ there exists $y\in X$ with $d(x,y)<\delta$ satisfying
\[
    \limsup_{n\to\infty} \frac{1}{H_n}\sum_{k=1}^{n}\frac{1}{k}d(T^{k-1}x,T^{k-1}y)\geq \eps.
\]
\end{defn}

Using Theorem \ref{thm:log-meqn-eq-main-result} and Lemma \ref{lem:log-mean-seq} we can restate Theorem~\ref{thm:dich-mean-eq} as the following dichotomy result.

\begin{prop} \label{prop:dich-log-mean-eq}
Let $(X,T)$ be a minimal system.
Then $(X,T)$ is either logarithmically mean equicontinuous or logarithmically mean sensitive.
\end{prop} 

Combining Theorem~\ref{thm:log-meqn-eq-main-result} and Proposition~\ref{prop:dich-log-mean-eq}, we have the following. 

\begin{coro}\label{cor:mean-sens}
Let $(X,T)$ be a minimal system.
Then $(X,T)$ is logarithmically mean sensitive
if and only if it is mean sensitive.
\end{coro}

\begin{rem}
It is interesting to know whether Corollary \ref{cor:mean-sens} holds for transitive systems.
\end{rem}

\section{Logarithmic mean equicontinuity with respect to an ergodic measure}

In \cite{HLY11} Huang et al.\  studied equicontinuity and sensitivity of a topological dynamical system with respect to invariant measures.
Following these ideas, García-Ramos studied in \cite{G17} the notions of mean equicontinuity and mean sensitivity relative to invariant measures.
Inspired by these results, we introduce logarithmic mean equicontinuity and logarithmic mean sensitivity with respect to a fixed invariant measure of a given topological dynamical system. 
We will show that when this measure is ergodic, logarithmic mean equicontinuity is equivalent to mean equicontinuity.

\begin{defn}
Let $(X,T)$ be a topological dynamical system and $\mu\in M_T(X)$.
We say that $(X,T)$ is \emph{$\mu$-mean equicontinuous} if for any $\eta\in (0,1)$ there exists a Borel subset $X_\eta$ of $X$ with $\mu(X_\eta)>\eta$ such that for any $\eps>0$ there exists  $\delta>0$ such that for any $x,y\in X_\eta$ with $d(x,y)<\delta$ one has 
\[
    \limsup_{n\to\infty} \frac{1}{n}\sum_{k=1}^{n}
    d(T^{k-1}x,T^{k-1}y)<\eps.
\]
We say that $(X,T)$ is \emph{$\mu$-mean sensitive} if there exists a constant $\eps>0$ such that for any Borel subset $B$ of $X$ with $\mu(B)>0$ there exist $x,y\in B$ such that 
\[
    \limsup_{n\to\infty} \frac{1}{n}\sum_{k=1}^{n}
    d(T^{k-1}x,T^{k-1}y)\geq \eps.
\]
\end{defn}

The following dichotomy result was proved in \cite[Theorem 3.6]{G17}.

\begin{thm} \label{thm:dict-mu-mean-eq}
Let $(X,T)$ be a topological dynamical system and $\mu\in M_T^{\mathrm{erg}}(X)$.
Then $(X,T)$ is either $\mu$-mean equicontinuous
or $\mu$-mean sensitive.
\end{thm}

Combining \cite[Theorem 2.7]{L16} and \cite[Remark 4.3]{L16}, we have the following characterization of $\mu$-mean sensitivity.
\begin{thm}  \label{thm:mu-mean-sensitive}
Let $(X,T)$ be a topological dynamical system and $\mu\in M_T^{\mathrm{erg}}(X)$.
If $(X,T)$ is $\mu$-mean sensitive, then  there exists a constant $\delta>0$ such that for $\mu\times\mu$-a.e.\ $(x,y)\in X\times X$ we have
\[
    \liminf_{n\to\infty} \frac{1}{n}\sum_{k=1}^{n}
    d(T^{k-1}x,T^{k-1}y)\geq \delta.
\]
\end{thm}

\begin{defn}
Let $(X,T)$ be a topological dynamical system and $\mu\in M_T(X)$.
We say that $(X,T)$ is \emph{$\mu$-logarithmically mean equicontinuous} if for any $\eta\in (0,1)$ there exists a Borel subset $X_\eta$ of $X$ with $\mu(X_\eta)>\eta$ such that for any $\eps>0$ there exists  $\delta>0$ such that for any $x,y\in X_\eta$ with $d(x,y)<\delta$, one has 
\[
    \limsup_{n\to\infty} \frac{1}{H_n}\sum_{k=1}^{n}\frac{1}{k}d(T^{k-1}x,T^{k-1}y)<\eps.
\]
We say that $(X,T)$ is \emph{$\mu$-logarithmically mean sensitive} if there exists a constant $\eps>0$ such that for any Borel subset $B$ of $X$ with $\mu(B)>0$ there exist $x,y\in B$ satisfying 
\[
    \limsup_{n\to\infty} \frac{1}{H_n}\sum_{k=1}^{n}\frac{1}{k}d(T^{k-1}x,T^{k-1}y)\geq \eps.
\]
\end{defn}

\begin{prop}\label{prop:mu-mean-sen}
Let $(X,T)$ be a topological dynamical system and $\mu\in M_T^{\mathrm{erg}}(X)$.
Then $(X,T)$ is $\mu$-mean sensitive if and only if
it is $\mu$-logarithmically mean sensitive.
\end{prop}
\begin{proof}
($\Leftarrow$)
By Lemma~\ref{lem:log-mean-seq}, 
for any $x,y\in X$ we have
\[ 
\limsup_{n\to\infty} \frac{1}{H_n}\sum_{k=1}^{n}\frac{1}{k}d(T^{k-1}x,T^{k-1}y)\leq \limsup_{n\to\infty} \frac{1}{n}\sum_{k=1}^{n}
    d(T^{k-1}x,T^{k-1}y).\]
Then it is clear that $(X,T)$ is $\mu$-mean sensitive provided that
it is $\mu$-logarithmically mean sensitive.

($\Rightarrow$)
Assume that $(X,T)$ is $\mu$-mean sensitive.
By Theorem~\ref{thm:mu-mean-sensitive}, 
there exists a constant $\eps>0$ such that for $\mu\times\mu$-a.e.\ $(x,y)\in X\times X$ it holds
\[
    \liminf_{n\to\infty} \frac{1}{n}\sum_{k=1}^{n}
    d(T^{k-1}x,T^{k-1}y)\geq \eps.
\]
Using Lemma~\ref{lem:log-mean-seq} again, 
we see that for any $x,y\in X$ it holds
\[ 
\liminf_{n\to\infty} \frac{1}{n}\sum_{k=1}^{n}
    d(T^{k-1}x,T^{k-1}y)\leq  \liminf_{n\to\infty} \frac{1}{H_n}\sum_{k=1}^{n}\frac{1}{k}d(T^{k-1}x,T^{k-1}y).\]
It follows that for $\mu\times\mu$-a.e.\ $(x,y)\in X\times X$ we have
\[
    \limsup_{n\to\infty} \frac{1}{H_n}\sum_{k=1}^{n}\frac{1}{k}d(T^{k-1}x,T^{k-1}y) \geq \liminf_{n\to\infty} \frac{1}{H_n}\sum_{k=1}^{n}\frac{1}{k}d(T^{k-1}x,T^{k-1}y)\geq \eps.
\]
Hence, $(X,T)$ is $\mu$-logarithmically mean sensitive.
\end{proof}

\begin{thm}\label{thm:mu-log-mean-eq}
Let $(X,T)$ be a topological dynamical system and $\mu\in M_T^{\mathrm{erg}}(X)$.
Then $(X,T)$ is $\mu$-mean equicontinuous if and only if
it is $\mu$-logarithmically mean equicontinuous.
\end{thm}
\begin{proof}
($\Rightarrow$)
By Lemma~\ref{lem:log-mean-seq}, it is clear that if $(X,T)$ is $\mu$-mean equicontinuous, then it is also $\mu$-logarithmically mean equicontinuous. 

($\Leftarrow$) 
If $(X,T)$ is not $\mu$-mean equicontinuous, then using Theorem~\ref{thm:dict-mu-mean-eq}  we get that $(X,T)$ is $\mu$-mean sensitive.
By the proof of Proposition~\ref{prop:mu-mean-sen},
there exists a constant $\eps>0$ such that for $\mu\times\mu$-a.e.\ pair $(x,y)\in X\times X$ we have
\begin{equation}\label{ineq:bad_pair}
    \limsup_{n\to\infty} \frac{1}{H_n}\sum_{k=1}^{n}\frac{1}{k}d(T^{k-1}x,T^{k-1}y) \geq \eps.
\end{equation}
 
Fix a Borel subset $B$ of $X$ with $\mu(B)>0$. Note that $\mu\times\mu(B\times B)>0$. Then for some $(x,y)\in B\times B$ with the same $\eps>0$ it holds
\[
\limsup_{n\to\infty} \frac{1}{H_n}\sum_{k=1}^{n}\frac{1}{k}d(T^{k-1}x,T^{k-1}y) \geq \eps.
\]
Hence, $(X,T)$ is not $\mu$-logarithmically mean equicontinuous.
\end{proof}

\begin{rem}
It is interesting to know whether Theorem~\ref{thm:mu-log-mean-eq} holds for an arbitrary (that is, not necessarily ergodic) invariant measure.
\end{rem}

\medskip

\noindent\textbf{Acknowledgements}: 
D.~Kwietniak was partially supported by the National Science Centre, Poland  under the Weave-UNISONO call in the Weave programme [grant no.~2021/03/Y/ST1/00072]. J. Li was partially supported by National Key R\&D Program of China (2024YFA1013601) and NSF of China (12222110, 12171298). H.~Pourmand was partially supported by the National Science Centre, Poland  under the grant no.~2022/47/P/ST1/00854.
The authors would like to thank the referee for the careful
reading and suggestions that helped to improved the manuscript.

\bibliographystyle{amsplain}

\end{document}